\documentclass[reqno, a4paper]{amsart}

\usepackage{amsthm}
\usepackage{amsmath}
\usepackage[all]{xy} \SelectTips{eu}{}
\usepackage{latexsym,amsfonts,amssymb}
\usepackage{hyperref}
\usepackage{mathptmx}
\usepackage{nicefrac}
\usepackage{array}
\usepackage{xcolor}

\newcommand{\numberseries}{\bfseries}   

\newlength{\thmtopspace}                
\newlength{\thmbotspace}                
\newlength{\thmheadspace}               
\newlength{\thmindent}                  

\newtheoremstyle{fixed bf head,slanted body}
                {\thmtopspace}{\thmbotspace}{\slshape}
                {\thmindent}{\bfseries}{.}{\thmheadspace}
                {{\numberseries \thmnumber{#2\;}}\thmname{#1}\thmnote{ (#3)}}

\newtheoremstyle{variable bf head,slanted body}
                {\thmtopspace}{\thmbotspace}{\slshape}
                {\thmindent}{\bfseries}{.}{\thmheadspace}
                {{\numberseries \thmnumber{#2\;}}\thmname{#1}\thmnote{ #3}}

\newtheoremstyle{fixed bf head,upright body}
                {\thmtopspace}{\thmbotspace}{\upshape}
                {\thmindent}{\bfseries}{.}{\thmheadspace}
                {{\numberseries \thmnumber{#2\;}}\thmname{#1}\thmnote{ (#3)}}

\newtheoremstyle{numbered paragraph}
                {\thmtopspace}{\thmbotspace}{\upshape}
                {\thmindent}{\upshape}{}{\thmheadspace}
                {{\numberseries \thmnumber{#2.}}}

\theoremstyle{fixed bf head,slanted body}
\newtheorem{res}{}[section]
\newtheorem{thm}[res]{Theorem}          \newtheorem*{thm*}{Theorem}
\newtheorem{prp}[res]{Proposition}      \newtheorem*{prp*}{Proposition}
        \newtheorem*{cor*}{Corollary}
\newtheorem{lem}[res]{Lemma}            \newtheorem*{lem*}{Lemma}

\theoremstyle{variable bf head,slanted body}
     \newtheorem*{introthm*}{Theorem}
   \newtheorem*{introcor*}{Corollary}

\theoremstyle{fixed bf head,upright body}
            \newtheorem*{stp*}{Setup}
       \newtheorem*{dfn*}{Definition}
\newtheorem{con}[res]{Construction}     \newtheorem*{con*}{Construction}
      \newtheorem*{obs*}{Observation}
\newtheorem{rmk}[res]{Remark}           \newtheorem*{rmk*}{Remark}
\newtheorem{exa}[res]{Example}          \newtheorem*{exa*}{Example}
         \newtheorem*{qst*}{Question}

\theoremstyle{numbered paragraph}
\newtheorem{ipg}[res]{}

\newlength{\thmlistleft}        
\newlength{\thmlistright}       
\newlength{\thmlistpartopsep}   
\newlength{\thmlisttopsep}      
\newlength{\thmlistparsep}      
\newlength{\thmlistitemsep}     

\newcounter{eqc} 
\newenvironment{eqc}{\begin{list}{\upshape (\textit{\roman{eqc}})}%
    {\usecounter{eqc}%
      \setlength{\leftmargin}{\thmlistleft}%
      \setlength{\labelwidth}{\thmlistleft}%
      \setlength{\rightmargin}{\thmlistright}%
      \setlength{\partopsep}{\thmlistpartopsep}%
      \setlength{\topsep}{\thmlisttopsep}%
      \setlength{\parsep}{\thmlistparsep}%
      \setlength{\itemsep}{\thmlistitemsep}}}%
  {\end{list}}%

\newcounter{prt}
\newenvironment{prt}{\begin{list}{\upshape (\alph{prt})}%
    {\usecounter{prt}%
      \setlength{\leftmargin}{\thmlistleft}%
      \setlength{\labelwidth}{\thmlistleft}%
      \setlength{\rightmargin}{\thmlistright}%
      \setlength{\partopsep}{\thmlistpartopsep}%
      \setlength{\topsep}{\thmlisttopsep}%
      \setlength{\parsep}{\thmlistparsep}%
      \setlength{\itemsep}{\thmlistitemsep}}}%
  {\end{list}}%

\newcounter{rqm}
  {\end{list}}%

  {\end{list}}%

\newcommand{\pgref}[1]{\ref{#1}}
\newcommand{\thmref}[2][Theorem~]{#1\pgref{thm:#2}}

\newcommand{\prpref}[2][Proposition~]{#1\pgref{prp:#2}}
\newcommand{\lemref}[2][Lemma~]{#1\pgref{lem:#2}}

\newcommand{\exaref}[2][Example~]{#1\pgref{exa:#2}}
\newcommand{\rmkref}[2][Remark~]{#1\pgref{rmk:#2}}
\newcommand{\secref}[2][Section~]{#1\ref{sec:#2}}

\renewcommand{\eqref}[1]{(\pgref{eq:#1})}

\makeatletter
\def\@nobreak@#1{\mathchoice%
  {\nobreakdef@\displaystyle\f@size{#1}}%
  {\nobreakdef@\nobreakstyle\tf@size{\firstchoice@false #1}}%
  {\nobreakdef@\nobreakstyle\sf@size{\firstchoice@false #1}}%
  {\nobreakdef@\nobreakstyle\ssf@size{\firstchoice@false #1}}%
  \check@mathfonts}%
\def\nobreakdef@#1#2#3{\hbox{{%
                    \everymath{#1}%
                    \let\f@size#2\selectfont%
                    #3}}}%
\makeatother

\DeclareFontFamily{T1}{cmex}{}
\DeclareFontShape{T1}{cmex}{m}{n}{<-> s * [0.89] cmex10}{}
\DeclareSymbolFont{cmlargesymbols}{T1}{cmex}{m}{n}
\DeclareMathSymbol{\mycoprod}{\mathop}{cmlargesymbols}{"60} 
\DeclareMathSymbol{\myprod}{\mathop}{cmlargesymbols}{"51} \let\prod\myprod

\DeclareSymbolFont{usualmathcal}{OMS}{cmsy}{m}{n}
\DeclareSymbolFontAlphabet{\mathcal}{usualmathcal}

\DeclareSymbolFont{letters}{OML}{txmi}{m}{it}

\DeclareMathSymbol{\alpha}{\mathord}{letters}{"0B}
\DeclareMathSymbol{\beta}{\mathord}{letters}{"0C}
\DeclareMathSymbol{\gamma}{\mathord}{letters}{"0D}
\DeclareMathSymbol{\delta}{\mathord}{letters}{"0E}
\DeclareMathSymbol{\epsilon}{\mathord}{letters}{"0F}
\DeclareMathSymbol{\zeta}{\mathord}{letters}{"10}
\DeclareMathSymbol{\eta}{\mathord}{letters}{"11}
\DeclareMathSymbol{\theta}{\mathord}{letters}{"12}
\DeclareMathSymbol{\iota}{\mathord}{letters}{"13}
\DeclareMathSymbol{\kappa}{\mathord}{letters}{"14}
\DeclareMathSymbol{\lambda}{\mathord}{letters}{"15}
\DeclareMathSymbol{\mu}{\mathord}{letters}{"16}
\DeclareMathSymbol{\nu}{\mathord}{letters}{"17}
\DeclareMathSymbol{\xi}{\mathord}{letters}{"18}
\DeclareMathSymbol{\pi}{\mathord}{letters}{"19}
\DeclareMathSymbol{\rho}{\mathord}{letters}{"1A}
\DeclareMathSymbol{\sigma}{\mathord}{letters}{"1B}
\DeclareMathSymbol{\tau}{\mathord}{letters}{"1C}
\DeclareMathSymbol{\upsilon}{\mathord}{letters}{"1D}
\DeclareMathSymbol{\phi}{\mathord}{letters}{"1E}
\DeclareMathSymbol{\chi}{\mathord}{letters}{"1F}
\DeclareMathSymbol{\psi}{\mathord}{letters}{"20}
\DeclareMathSymbol{\omega}{\mathord}{letters}{"21}
\DeclareMathSymbol{\varepsilon}{\mathord}{letters}{"22}
\DeclareMathSymbol{\vartheta}{\mathord}{letters}{"23}
\DeclareMathSymbol{\varpi}{\mathord}{letters}{"24}
\DeclareMathSymbol{\varrho}{\mathord}{letters}{"25}
\DeclareMathSymbol{\varsigma}{\mathord}{letters}{"26}
\DeclareMathSymbol{\varphi}{\mathord}{letters}{"27}

\DeclareMathSymbol{\Gamma}{\mathord}{letters}{"00}
\DeclareMathSymbol{\Delta}{\mathord}{letters}{"01}
\DeclareMathSymbol{\Theta}{\mathord}{letters}{"02}
\DeclareMathSymbol{\Lambda}{\mathord}{letters}{"03}
\DeclareMathSymbol{\Xi}{\mathord}{letters}{"04}
\DeclareMathSymbol{\Pi}{\mathord}{letters}{"05}
\DeclareMathSymbol{\Sigma}{\mathord}{letters}{"06}
\DeclareMathSymbol{\Upsilon}{\mathord}{letters}{"07}
\DeclareMathSymbol{\Phi}{\mathord}{letters}{"08}
\DeclareMathSymbol{\Psi}{\mathord}{letters}{"09}
\DeclareMathSymbol{\Omega}{\mathord}{letters}{"0A}

\DeclareMathSymbol{\upGamma}{\mathalpha}{operators}{"00}
\DeclareMathSymbol{\upDelta}{\mathalpha}{operators}{"01}
\DeclareMathSymbol{\upTheta}{\mathalpha}{operators}{"02}
\DeclareMathSymbol{\upLambda}{\mathalpha}{operators}{"03}
\DeclareMathSymbol{\upXi}{\mathalpha}{operators}{"04}
\DeclareMathSymbol{\upPi}{\mathalpha}{operators}{"05}
\DeclareMathSymbol{\upSigma}{\mathalpha}{operators}{"06}
\DeclareMathSymbol{\upUpsilon}{\mathalpha}{operators}{"07}
\DeclareMathSymbol{\upPhi}{\mathalpha}{operators}{"08}
\DeclareMathSymbol{\upPsi}{\mathalpha}{operators}{"09}
\DeclareMathSymbol{\upOmega}{\mathalpha}{operators}{"0A}

\renewcommand{\con}[1]{\textnormal{({\small #1})}}
\newcommand{\ideal}[1]{\text{\small$($}#1\text{\small$)$}}
\newcommand{\idealt}[1]{\text{\tiny$($}#1\text{\tiny$)$}}
\newcommand{\Ltp}[3][R]{#2\otimes_{#1}^\mathbf{L}#3}
\newcommand{\Gdim}[2]{\textnormal{G-dim}_{#1}(#2)}
\newcommand{\pd}[2]{\mathrm{pd}_{#1}(#2)}

\newcommand{\Tor}[4]{\operatorname{Tor}^{\mspace{2mu}#1}_{#2}(#3,#4)}
\newcommand{\Stor}[4][R]{\smash{\operatorname{\widetilde{Tor}}}_{#2}^{\mspace{3mu}{#1}^{\phantom{|\mspace{-6mu}}}}(#3,#4)}

\begin{document}

\vspace*{-0.1ex}

\title{On modules with self Tor vanishing}
\author{Olgur Celikbas \ }
\address{Department of Mathematics, West Virginia University, Morgantown, WV 26506 U.S.A} 
\email{olgur.celikbas@math.wvu.edu}

\author{ \ Henrik Holm}
\address{Department of Mathematical Sciences, Universitetsparken 5, University of Co\-penhagen, 2100 Copenhagen {\O}, Denmark} 
\email{holm@math.ku.dk}
\urladdr{http://www.math.ku.dk/\~{}holm/}





\begin{abstract} 
The long-standing Auslander and Reiten Conjecture states that a finitely generated module over a finite-dimensional algebra is projective if certain Ext-groups vanish. Several authors, including Avramov, Buchweitz, Iyengar, Jorgensen, Nasseh, Sather-Wagstaff, and \c{S}ega, have studied a possible counterpart of the conjecture, or question, for commutative rings in terms of vanishing of Tor. This has led to the notion of Tor-persistent rings. Our main result shows that the class of Tor-persistent local rings is closed under a number of standard procedures in ring theory.
\end{abstract}
\maketitle


\section{Introduction}
\label{sec:Introduction}

Inspired by work of \c{S}ega \cite[para.~preceding Thm.~2.6]{MR2769231}, Avramov, Iyengar, Nasseh, and Sather-Wagstaff raise in \cite{AINAW2}\footnote{Note that this work is announced under the different title
\emph{Vanishing of endohomology over local rings} in \cite{AINAW}.}, 
the question of whether every commutative noetherian ring is Tor-persistent. A commutative ring $A$ is said to be \emph{Tor-persistent} if every finitely generated $A$-module $M$ with $\Tor{A}{i}{M}{M}=0$ for all $i\gg 0$, that is, $\Tor{A}{}{M}{M}$ is bounded, has finite projective dimension. We refer to \cite{AINAW2} and the precursor \cite{AINAW} (by the same authors) for a history/background of this question. The mentioned works also contain information about several interesting classes of rings which are known to be Tor-persistent. This includes Gorenstein rings with an exact zero divisor whose radical to the fourth power is zero \cite[Thm.~2]{MR2769231}, complete intersection rings \cite[Cor.~(1.2)]{DAJ97} (see also \cite[Thm.~IV]{LLAROB00} and \cite[Thm.~1.9]{MR1612887}) and Golod rings \cite[Thm.~3.1]{DAJ99}.

   In \cite[Prop.~1.6]{AINAW2} it is shown that a commutative noethe\-rian ring $A$ is Tor-persistent if and only if the localization $A_\mathfrak{m}$ is so for every maximal ideal $\mathfrak{m} \subset A$; hence it suffices to study the question mentioned above for commutative noetherian \textsl{local} rings. Throughout this paper, $(R,\mathfrak{m},k)$ denotes such a ring. Our main result is the following:

\begin{thm} 
  \label{thm:mainresult}
  The following conditions are equivalent:
  \begin{eqc}
  \item $R$ is Tor-persistent.
  \item $\widehat{R}$ is Tor-persistent.
  \item \mbox{$R[\mspace{-2.5mu}[X_1,\ldots,X_n]\mspace{-2.5mu}]$} is Tor-persistent.
  \item $R[X_1,\ldots,X_n]_{(\mathfrak{m},X_1,\ldots,X_n)}$ is Tor-persistent.
  \end{eqc}
\end{thm}

While some papers in the literature approach the question raised in \cite{AINAW2} by finding specific conditions that imply Tor-persistence, we show that Tor-persistence is a property preserved by standard procedures in local algebra. Our work is motivated by \cite{LWCHHl12} where a result similar to \thmref{mainresult} is proved for the so-called Auslander's condition. However, our arguments are somewhat different since the techniques used in \emph{loc.~cit.}~do not work in our setting; see \rmkref{question} and \cite[Cor.~(2.2)]{LWCHHl12}.

It should be noticed that there is some overlap between this paper and \cite{AINAW2}. For example, the equivalence \mbox{($i$) $\Leftrightarrow$ ($ii$)} in \thmref{mainresult}  is contained in \cite[Prop.~1.5]{AINAW2}, and our \prpref[Prop\-osi\-tion~]{regular} is akin to \cite[Prop.~3.8]{AINAW2}. However, the two papers have been written completely independently, indeed, \cite{AINAW2} were only made available to us after we completed this work. Subsequently, we rewrote our introduction and adopted the terminology ``Tor-persistent'' coined in \cite{AINAW2}.

This short paper is organized as follows. In \secref{Main} we prove \thmref{mainresult} and show how to construct new examples of Tor-persistent rings (Example \ref{exa:HSV}). We also give a way to obtain certain kinds of regular sequences in power series rings (\lemref{sequence}),~which~might be of independent interest. In \secref{G} we consider another property for rings, called \con{TG}; it is a slightly weaker property than Tor-persistence and it is related to the Gorenstein dimension. For this property we prove a result similar to \thmref{mainresult} (see \thmref{TG}), and show that some results from \secref{Main} can be strengthened in this new setting.

\section{Main results}
\label{sec:Main}

\begin{lem}
  \label{lem:fd}
   Let $(R,\mathfrak{m},k) \to (S,\mathfrak{n},\ell)$ be a local homomorphism of commutative noetherian local rings. If $S$ is Tor-persistent and has finite flat dimension over $R$, then $R$ is Tor-persistent.
\end{lem}

\begin{proof}
  Assume $S$ is Tor-persistent and let $M$ be a finitely generated $R$-module such that $\Tor{R}{i}{M}{M}=0$ for all $i\gg 0$. We have $\Tor{R}{i}{M}{S}=0$ for each $i>d$, where $d$ is the flat dimension of $S$ over $R$. Replacing $M$ by a sufficiently high syzygy we can (by dimension shifting) assume that $\Tor{R}{i}{M}{M}=0$ and $\Tor{R}{i}{M}{S}=0$ for every $i>0$. In this case there is an isomorphism $\Ltp{M}{S} \cong M \otimes_R S$ in the derived category over $S$. This yields:
\begin{displaymath}  
  \Ltp{(\Ltp{M}{M})}{S} 
  \cong
  \Ltp[S]{(\Ltp{M}{S})}{(\Ltp{M}{S})}
  \cong
  \Ltp[S]{(M \otimes_R S)}{(M \otimes_R S)}\;.
\end{displaymath}    
As the complex $\Ltp{M}{M}$ is homologically bounded (its homology is even concentrated in degree zero) and since $S$ has finite flat dimension over $R$, the left-hand side is homologi\-cally bounded, and hence so is the right-hand side. That is, $\Tor{S}{i}{M \otimes_R S}{M \otimes_R S}=0$ for all $i\gg 0$. As $S$ is Tor-persistent, it follows that $M \otimes_R S \cong \Ltp{M}{S}$ has finite projective dimension over $S$. It follows from \cite[(1.5.3)]{MR1455856} that $\pd{R}{M}$ is finite.
\end{proof}

\begin{prp}
  \label{prp:regular}
  Let $(R,\mathfrak{m},k)$ be a commutative noetherian local ring and let $\underline{x}=x_1,\ldots,x_n$ be an $R$-regular sequence. If $R/\ideal{\underline{x}}$ is Tor-persistent, then $R$ is Tor-persistent. The converse is true if $x_i \notin \mathfrak{m}^2 + \ideal{x_1,\ldots,x_{i-1}}$ holds for every $i=1,\ldots,n$.
\end{prp}

\begin{proof}
    The first statement is a special case of \lemref{fd}. We now prove the (partial) converse. By assumption, $\bar{x}_i$ is a non zero-divisor on $R/\ideal{x_1,\ldots,x_{i-1}}$, which has the maximal ideal  $\bar{\mathfrak{m}} = \mathfrak{m}/\ideal{x_1,\ldots,x_{i-1}}$. Since $x_i \notin \mathfrak{m}^2 + \ideal{x_1,\ldots,x_{i-1}}$ we have $\bar{x}_i \notin \bar{\mathfrak{m}}^2$, so by induction it suffices to consider the case where $n=1$.
    
  Let $R$ be Tor-persistent and let $x \in \mathfrak{m} \smallsetminus \mathfrak{m}^2$ be a non zero-divisor on $R$. To see that $R/\ideal{x}$ is Tor-persistent, let $N$ be a finitely generated $R/\ideal{x}$-module with \smash{$\Tor{R/\idealt{x}}{i}{N}{N}=0$} for all $i\gg 0$. By \cite[11.65]{rot} (see also \cite[Lem.~2.1]{MR1282227}) there is a long exact sequence,
\begin{displaymath}
  \cdots \longrightarrow \Tor{R/\idealt{x}}{i-1}{N}{N} 
  \longrightarrow \Tor{R}{i}{N}{N} 
  \longrightarrow \Tor{R/\idealt{x}}{i}{N}{N} \longrightarrow \cdots\;.
\end{displaymath}
Therefore $\Tor{R}{i}{N}{N}=0$ for all $i\gg 0$. Since $R$ is Tor-persistent, we get that $\pd{R}{N}$ is finite. As $x \notin \mathfrak{m}^2$, it follows that $\pd{R/\idealt{x}}{N}$ is finite; see e.g.~\cite[Prop.~3.3.5(1)]{ifr}.
\end{proof}

\begin{rmk}
  \label{rmk:question}
  It would be interesting to know if the last assertion in \prpref{regular} holds without the assumption $x_i \notin \mathfrak{m}^2 + (x_1,\ldots,x_{i-1})$, i.e.~if Tor-persistence is preserved when passing to the quotient by an ideal generated by \textsl{any} regular sequence; cf.~\prpref{regular-G}.
\end{rmk}

\begin{rmk}
  \label{rmk:pows}
  The sequence $X_1,\ldots,X_n$ is regular on \mbox{$R[\mspace{-2.5mu}[X_1,\ldots,X_n]\mspace{-2.5mu}]$} and $X_i$ does not belong to $(\mathfrak{m},X_1,\ldots,X_n)^2+(X_1,\ldots,X_{i-1})$. It follows from \prpref{regular} that $R$ is Tor-persistent if and only if \mbox{$R[\mspace{-2.5mu}[X_1,\ldots,X_n]\mspace{-2.5mu}]$} is Tor-persistent.
\end{rmk}

\prpref{regular} can be used to construct new examples of Tor-persistent rings from known examples; see \exaref{HSV}. However, to do so it is useful to have a concrete way of constructing regular sequences with the property mentioned in \prpref[]{regular}. In \lemref{sequence} below we give one such construction.

If $A$ is a commutative ring and $a$ is an element in $A$, then it can happen, perhaps surprisingly, that $X-a$ is a zero-divisor on \mbox{$A[\mspace{-2.5mu}[X]\mspace{-2.5mu}]$}; see \cite[p.~146]{MR0387274} for an example. However, as is well-known, if $A$ is noetherian, then the situation is much nicer:

\begin{ipg}
  \label{nzd}
  Let $A$ be a commutative noetherian ring and consider an elemement $f=f(X_1,\ldots,X_n)$ in \mbox{$A[\mspace{-2.5mu}[X_1,\ldots,X_n]\mspace{-2.5mu}]$}. It follows from \cite[Thm.~5]{MR0271100} that if $f$ has \textsl{some} coefficient which is a unit in $A$, then $f$ is a non zero-divisor on \mbox{$A[\mspace{-2.5mu}[X_1,\ldots,X_n]\mspace{-2.5mu}]$}.
\end{ipg}

\begin{lem}
  \label{lem:sequence}
  Let $(R,\mathfrak{m},k)$ be a commutative noetherian local ring. Consider the power series ring \mbox{$S=R[\mspace{-2.5mu}[X_1,\ldots,X_n]\mspace{-2.5mu}]$} and write $\mathfrak{n} = (\mathfrak{m},X_1,\ldots,X_n)$ for its unique maximal ideal. Let $0=m_0 < m_1 < \cdots < m_{t-1} < m_t = n$ be integers and let $f_1,\ldots,f_t \in \mathfrak{n}$ be elements such that, for every $i=1,\ldots,t$, the following conditions hold:
  \begin{prt}
  \item \mbox{$f_i \in R[\mspace{-2.5mu}[X_1,\ldots,X_{m_i}]\mspace{-2.5mu}] \subseteq S$}.
  
  \item The element \smash{$\frac{\partial f_i}{\partial X_{\mspace{-2mu}j}}(0,\ldots,0) \in R$} is a unit for some $m_{i-1}<j$.
  \end{prt}
  Then $f_1,\ldots,f_t$ is a regular sequence on \mbox{$R[\mspace{-2.5mu}[X_1,\ldots,X_n]\mspace{-2.5mu}]$} with
  $f_i \notin \mathfrak{n}^2 + (f_1,\ldots,f_{i-1})$ for all $i$.
\end{lem}

\begin{proof}
  First note that condition (b) implies:
  \begin{equation}
    \label{eq:star}
    \textnormal{The power series $f_i(0,\ldots,0,X_{m_{i-1}+1},\ldots,X_n)$ has a coefficient which is a unit in $R$.}
  \end{equation}
  Indeed, if $m_{i-1}<j$, then \smash{$\frac{\partial f_i}{\partial X_{\mspace{-2mu}j}}(0,\ldots,0)$} \textsl{is} a coefficient in $f_i(0,\ldots,0,X_{m_{i-1}+1},\ldots,X_n)$.

  Next we show that $f_1,\ldots,f_t$ is a regular sequence. With $i=1$ condition \eqref{star} says that $f_1(X_1,\ldots,X_n)$ has a coefficient which is a unit in $R$, and so $f_1$ is a non zero-divisor on $S$ by \ref{nzd}. Next we show that $f_{i+1}$ is a non zero-divisor on $S/(f_1,\ldots,f_{i})$ where $i \geqslant 1$. Write
\begin{equation}
  \label{eq:f}
  f_{i+1} \,= \ \textstyle \sum_{v_{m_i+1},\ldots,v_{n}} \mspace{1mu} h_{v_{m_i+1},\ldots,v_n} \mspace{3mu} X_{m_i+1}^{v_{m_i+1}}\cdots X_n^{v_{n_{}}} \,\in\, S \cong
  R[\mspace{-2.5mu}[X_1,\ldots,X_{m_{i}}]\mspace{-2.5mu}] [\mspace{-2.5mu}[X_{m_{i}+1},\ldots,X_n]\mspace{-2.5mu}]
\end{equation}
with \mbox{$h_* \in R[\mspace{-2.5mu}[X_1,\ldots,X_{m_{i}}]\mspace{-2.5mu}]$}. 
As $f_1,\ldots,f_{i} \in R[\mspace{-2.5mu}[X_1,\ldots,X_{m_{i}}]\mspace{-2.5mu}]$ by (a) there is an isomorphism:
\begin{equation}
  \label{eq:iso}
  S/(f_1,\ldots,f_{i}) \,\cong\, \big(R[\mspace{-2.5mu}[X_1,\ldots,X_{m_{i}}]\mspace{-2.5mu}]/(f_1,\ldots,f_{i})\big) [\mspace{-2.5mu}[X_{m_{i}+1},\ldots,X_n]\mspace{-2.5mu}]\;.
\end{equation}  
In particular, the image $\bar{f}_{i+1}$ of $f_{i+1}$ in $S/(f_1,\ldots,f_{i})$ can be identified with the element
\begin{displaymath}
  \bar{f}_{i+1} \,= \ \textstyle \sum_{v_{m_i+1},\ldots,v_{n}} \mspace{1mu} \tilde{h}_{v_{m_i+1},\ldots,v_n} \mspace{3mu} X_{m_i+1}^{v_{m_i+1}}\cdots X_n^{v_{n_{}}}
\end{displaymath}
in the right-hand side of \eqref{iso}, where \smash{$\tilde{h}_*$} is the image of $h_*$ in $R[\mspace{-2.5mu}[X_1,\ldots,X_{m_{i}}]\mspace{-2.5mu}]/(f_1,\ldots,f_{i})$. Hence, to show that $\bar{f}_{i+1}$ is a non zero-divisor, it suffices by \ref{nzd} 
to argue that one of the~coefficients $\tilde{h}_*$ is a unit. By \eqref{star} we know that \mbox{$f_{i+1}(0,\ldots,0,X_{m_{i}+1},\ldots,X_n)$} has a coefficient which is a unit in $R$, and by \eqref{f} this means that one of the elements \smash{$h_{v_{m_i+1},\ldots,v_n}(0,\ldots,0) \in R$} is a unit. Consequently $h_{v_{m_i+1},\ldots,v_n} = h_{v_{m_i+1},\ldots,v_n}(X_1,\ldots,X_{m_i})$ will be a unit in $R[\mspace{-2.5mu}[X_1,\ldots,X_{m_{i}}]\mspace{-2.5mu}]$, so its image \smash{$\tilde{h}_{v_{m_i+1},\ldots,v_n}$}
  is also a unit, as desired.

Next we show that $f_i \notin \mathfrak{n}^2 + (f_1,\ldots,f_{i-1})$ holds for all $i$. Suppose for contradiction that:
\begin{displaymath}
\textstyle  f_i \,=\, \sum_{v} p_v q_v \,+\, \sum_{w=1}^{i-1}g_w f_w\,, \text{ where } p_v, q_v \in \mathfrak{n} \text{ and } g_w \in S.
\vspace*{0.75ex}
\end{displaymath}
By assumption (b) we have that \smash{$\frac{\partial f_i}{\partial X_{\mspace{-2mu}j}}(0,\ldots,0) \in R$} is a unit for some \mbox{$m_{i-1}<j$}. It follows from the identity above that:
\begin{displaymath}
\textstyle
\frac{\partial f_i}{\partial X_{\mspace{-2mu}j}}(\underline{0}) \,=\, 
\sum_{v} \Big( \frac{\partial p_v}{\partial X_{\mspace{-2mu}j}}(\underline{0})\, q_v(\underline{0}) + p_v(\underline{0})\, \frac{\partial q_v}{\partial X_{\mspace{-2mu}j}}(\underline{0})\Big) \,+\, 
\sum_{w=1}^{i-1} \Big( \frac{\partial g_w}{\partial X_{\mspace{-2mu}j}}(\underline{0})\, f_w(\underline{0}) + g_w(\underline{0})\, \frac{\partial f_w}{\partial X_{\mspace{-2mu}j}}(\underline{0})\Big)\;.
\end{displaymath}
As already mentioned, the left-hand side is a unit, and this contradicts that the right-hand side belongs to $\mathfrak{m}$. Indeed, we have $p_v(\underline{0}), q_v(\underline{0}), f_w(\underline{0}) \in \mathfrak{m}$ as $p_v, q_v, f_w \in \mathfrak{n}$. Furthermore, $f_1,\ldots,f_{i-1}$ only depend on the variables $X_1,\ldots,X_{m_{i-1}}$ by (a), so every \smash{$\frac{\partial f_w}{\partial X_{\mspace{-2mu}j}}$} is zero.
\end{proof}

\begin{exa}
  \label{exa:HSV}
  In \mbox{$R[\mspace{-2.5mu}[U,V,W]\mspace{-2.5mu}]$} the following (more or less arbitrarily chosen) sequence, corresponding to $t=2$ and $m_1=2$, satisfies the assumptions of \lemref{sequence}:
\begin{displaymath}
  f_1 = a + U^3 \mspace{-1mu}+\mspace{-1mu} UV \mspace{-1mu}+\mspace{-1mu} V \quad \textnormal{and} \quad
  f_2 = b + UV^2 \mspace{-1mu}+\mspace{-1mu} W \mspace{-1mu}+\mspace{-1mu} W^2 \qquad (a,b \in \mathfrak{m})\;.
\end{displaymath}
Indeed, (a) is clear and (b) holds since \smash{$\frac{\partial f_1}{\partial V}(0,0,0) = 1=\frac{\partial f_2}{\partial W}(0,0,0)$}. So \prpref{regular} implies that if $R$ is Tor-persistent, then so is \mbox{$A=R[\mspace{-2.5mu}[U,V,W]\mspace{-2.5mu}]/(f_1,f_2)$}. 

Note that the fiber product ring
\begin{displaymath}
  R \,=\, k[\mspace{-2.5mu}[X]\mspace{-2.5mu}]/(X^4) \,\times_k\, k[\mspace{-2.5mu}[Y]\mspace{-2.5mu}]/(Y^3) \,\cong\, k[\mspace{-2.5mu}[X,Y]\mspace{-2.5mu}]/(X^4,\,Y^3,\,XY)
\end{displaymath}
is artinian, not Gorenstein, and by \cite[Thm.~1.1]{MR3691985} it is Tor-persistent. Hence the following ring (where we have chosen $a=Y^2$ and $b=X^2$) is Tor-persistent as well:
\begin{displaymath}
\pushQED{\qed} 
  A = k[\mspace{-2.5mu}[X,Y,U,V,W]\mspace{-2.5mu}]/(X^4,\, Y^3,\, XY,\, Y^2 \mspace{-1mu}+\mspace{-1mu} U^3 \mspace{-1mu}+\mspace{-1mu} UV \mspace{-1mu}+\mspace{-1mu} V,\ X^2 \mspace{-1mu}+\mspace{-1mu} UV^2 \mspace{-1mu}+\mspace{-1mu} W+W^2)\;. 
\qedhere
\popQED
\end{displaymath}
\end{exa}

\begin{proof}[Proof of \thmref{mainresult}]
  The equivalence ($i$)\,$\Leftrightarrow$\,($iii$) is noted in \rmkref{pows}. Let $a_1,\ldots,a_n$ be a set of elements that generate $\mathfrak{m}$. We have \smash{$\widehat{R} \cong R[\mspace{-2.5mu}[X_1,\ldots,X_n]\mspace{-2.5mu}]/(X_1-a_1,\ldots,X_n-a_n)$} by \cite[Thm.~8.12]{Mat}. The sequence $f_i = X_i-a_i$ clearly satisfies the assumptions in \lemref{sequence}, so the equivalence ($i$)\,$\Leftrightarrow$\,($ii$) follows. Note that $R[X_1,\ldots,X_n]_{(\mathfrak{m},X_1,\ldots,X_n)}$ and \mbox{$R[\mspace{-2.5mu}[X_1,\ldots,X_n]\mspace{-2.5mu}]$} have isomorphic completions (both are isomorphic to \mbox{$\smash{\widehat{R}}[\mspace{-2.5mu}[X_1,\ldots,X_n]\mspace{-2.5mu}]$}), so the equivalence ($iii$)\,$\Leftrightarrow$\,($iv$) follows from the already established equivalence between ($i$) and ($ii$).
\end{proof}

\section{Connections with the Gorenstein dimension}
\label{sec:G}

In this section, we give a few remarks and observations pertaining Aulander's G-dimen\-sion \cite{MAs67} and self Tor vanishing. For a commutative noetherian local ring $(R,\mathfrak{m},k)$, we~consider the following property (which $R$ may, or may not, have):
\begin{prt}
  \item[\con{TG}] Every finitely generated $R$-module $M$ satisfying $\Tor{R}{i}{M}{M}=0$ for all $i\gg 0$ has finite G-dimension, that is, $\Gdim{R}{M} < \infty$.
\end{prt}

  Every Tor-persistent ring has the property \con{TG}, see \cite[Prop.~(1.2.10)]{LNM}, and the converse holds if the maximal ideal $\mathfrak{m}$ is decomposable; see \cite[Thm.~5.5]{NassehTakahashi2016}.

Testing finiteness of the G-dimension via the vanishing of Tor, in some form, is an idea pursued in a number of papers. For example, in \cite[Thm.~3.11]{MR3695854} it was proved that a finitely generated module $M$ over a commutative noetherian ring $R$ has finite G-dimension if and only if the stable homology \smash{$\Stor{i}{M}{R}$} vanishes for every $i\in \mathbb{Z}$. Furthermore, finitely generated modules testing finiteness of the G-dimension via the vanishing of absolute homology, i.e.~Tor, were also examined in \cite{MR3536062}.

For the property \con{TG} we have the following stronger version of   \prpref{regular}.

\begin{prp}
  \label{prp:regular-G}
  Let $(R,\mathfrak{m},k)$ be a commutative noetherian local ring and let $\underline{x}=x_1,\ldots,x_n$ be an $R$-regular sequence. Then $R$ has the property \con{TG} if and only if $R/\ideal{\underline{x}}$ has it.
\end{prp}

\begin{proof}
  For the ``if'' part we proceed as in the proof of \lemref{fd} with $S=R/\ideal{\underline{x}}$. Note that having replaced $M$ with a sufficiently high syzygy, the sequence $\underline{x}$ becomes regular on $M$ (this is standard but see also \cite[Lem.~5.1]{NassehTakahashi}). From the finiteness of $\Gdim{R/\idealt{\underline{x}}}{M/\ideal{\underline{x}}M}$ we infer the finiteness of $\Gdim{R}{M}$ from \cite[Cor.~(1.4.6)]{LNM}. For the ``only if'' part proceed as in the proof of \prpref{regular}. From the finiteness of $\Gdim{R}{N}$ one always gets finiteness of
$\Gdim{R/\idealt{x}}{N}$ (the assumption $x \notin \mathfrak{m}^2$ is not needed)~by~\cite[Thm.~p.~39]{LNM}.~\qedhere 
\end{proof}

Now the arguments in the proof of \thmref{mainresult} applies and give the following.

\enlargethispage{2ex}

\begin{thm}
  \label{thm:TG}
  Let $(R,\mathfrak{m},k)$ be a commutative noetherian local ring. The following conditions are equivalent:
  \begin{eqc}
  \item $R$ has the property \con{TG}.
  \item $\widehat{R}$ has the property \con{TG}.
  \item \mbox{$R[\mspace{-2.5mu}[X_1,\ldots,X_n]\mspace{-2.5mu}]$} has the property \con{TG}.
  \item $R[X_1,\ldots,X_n]_{(\mathfrak{m},X_1,\ldots,X_n)}$ has the property \con{TG}. \qed
  \end{eqc}
\end{thm}

\section*{Acknowledgments}

We thank Avramov, Iyengar, Nasseh, and Sather-Wagstaff for useful comments and for making their manuscript \cite{AINAW2} available to us.

Part of this work was completed when Holm visited West Virginia University in March 2018. He is grateful for the kind hospitality of the WVU Department of Mathematics.


\def\cprime{$'$} \def\soft#1{\leavevmode\setbox0=\hbox{h}\dimen7=\ht0\advance
  \dimen7 by-1ex\relax\if t#1\relax\rlap{\raise.6\dimen7
  \hbox{\kern.3ex\char'47}}#1\relax\else\if T#1\relax
  \rlap{\raise.5\dimen7\hbox{\kern1.3ex\char'47}}#1\relax \else\if
  d#1\relax\rlap{\raise.5\dimen7\hbox{\kern.9ex \char'47}}#1\relax\else\if
  D#1\relax\rlap{\raise.5\dimen7 \hbox{\kern1.4ex\char'47}}#1\relax\else\if
  l#1\relax \rlap{\raise.5\dimen7\hbox{\kern.4ex\char'47}}#1\relax \else\if
  L#1\relax\rlap{\raise.5\dimen7\hbox{\kern.7ex
  \char'47}}#1\relax\else\message{accent \string\soft \space #1 not
  defined!}#1\relax\fi\fi\fi\fi\fi\fi} \def\cprime{$'$}
  \providecommand{\arxiv}[2][AC]{\mbox{\href{http://arxiv.org/abs/#2}{\sf
  arXiv:#2 [math.#1]}}}
  \providecommand{\oldarxiv}[2][AC]{\mbox{\href{http://arxiv.org/abs/math/#2}{\sf
  arXiv:math/#2
  [math.#1]}}}\providecommand{\MR}[1]{\mbox{\href{http://www.ams.org/mathscinet-getitem?mr=#1}{#1}}}
  \renewcommand{\MR}[1]{\mbox{\href{http://www.ams.org/mathscinet-getitem?mr=#1}{#1}}}
\providecommand{\bysame}{\leavevmode\hbox to3em{\hrulefill}\thinspace}
\providecommand{\MR}{\relax\ifhmode\unskip\space\fi MR }
\providecommand{\MRhref}[2]{%
  \href{http://www.ams.org/mathscinet-getitem?mr=#1}{#2}
}
\providecommand{\href}[2]{#2}

\end{document}